\documentclass[review]{elsarticle}

\usepackage{amssymb}
\usepackage{amsmath}
\usepackage{amsthm}
\usepackage{amsfonts}
\usepackage{bm}
\usepackage{mathrsfs}
\usepackage{curves}
\usepackage{color}
\usepackage{subfig}
\usepackage{booktabs}
\usepackage{bookmark}
\usepackage{multirow}

\usepackage{lineno,hyperref}
\modulolinenumbers[5]

\newtheorem{theorem}{Theorem}[section]
\newtheorem{lemma}{Lemma}[section]

\newdefinition{definition}{Definition}[section]
\newdefinition{example}{Example}[section]
\numberwithin{equation}{section}

\begin{document}
\begin{frontmatter}

\title{Numerical algorithm for two-dimensional time-fractional wave equation of distributed-order with a nonlinear source term}
\author[nwpu,hut]{Jiahui Hu}
\ead{hujh@mail.nwpu.edu.cn}
\author[nwpu]{Jungang Wang}
\author[nwpu]{Zhanbin Yuan}
\author[nwpu]{Zongze Yang}
\author[nwpu]{Yufeng Nie\corref{cor1}}
\ead{yfnie@nwpu.edu.cn}

\cortext[cor1]{Corresponding author}
\address[nwpu]{Research Center for Computational Science, Northwestern Polytechnical University, Xi'an 710129, China}
\address[hut]{College of Science, Henan University of Technology, Zhengzhou 450001, China}
\begin{abstract}

  In this paper, an alternating direction implicit (ADI) difference scheme for two-dimensional
  time-fractional wave equation of distributed-order with a nonlinear source term is presented.
  The unique solvability of the difference solution is discussed, and
  the unconditional stability and convergence order of the numerical scheme are analysed.
  Finally, numerical experiments are carried out to verify the effectiveness and accuracy of the algorithm.

\end{abstract}
\begin{keyword}
  Two-dimensional time-fractional wave equation of distributed-order \sep
  ADI scheme \sep
  Nonlinear source term \sep
  Stability \sep
  Convergence
  \MSC[2010] 35R11\sep 65M06\sep 65M12
\end{keyword}
\end{frontmatter}


\section[Introduction]{Introduction}\label{sec:introduciton}
The idea of distributed-order differential equation was first introduced by Caputo in his work for modeling
the stress-strain behavior of an anelastic medium in 1960s \cite{Caputo}.
Being different from the differential equations with the single-order fractional derivative and
the ones with sums of fractional derivatives, i.e., multi-term fractional differential equations (FDEs), the
distributed-order differential equations are derived by integrating the order of differentiation over a certain range.
It can be regarded as a generalization of the aforementioned two classes of FDEs.
A typical application of this kind of FDEs is in the retarding sub-diffusion process, where
a plume of particles spreads at a logarithmic rate, which leads to ultraslow diffusion
(see \cite{Sinai1983}\cite{ChechkinKlafterSokolov2003}\cite{Kochubei2008}).
Another example is
the fractional Langevin equation of distributed-order, which was proposed to model the kinetics of retarding sub-diffusion
whose scaling exponent decreases in time,
and then was applied to simulate
the strongly anomalous ultraslow diffusion with the mean square displacement growing as a power of logarithm of time \cite{EabLim2011}.
The distributed-order FDEs were also found playing important role in other various research fields, such as
control and signal processing \cite{JiaoChenPodlubny2012}, modelling dielectric induction and diffusion \cite{Caputo2001}, identification of
systems \cite{Hartley1999}, and so on.

Till now, there have been many important progresses for the research on analytical solutions of distributed-order FDEs.
For the kinetic description of anomalous
diffusion and relaxation phenomena,
A. V. Chechkin et al. presented the diffusion-like equation with time fractional derivative of distributed-order
in \cite{Chechkin2003}, where the positivity of the solutions of the proposed equation was proved and the relation to
the continuous-time random walk theory was established.
T. M. Atanackovic et al. analysed a Cauchy
problem for a time distributed-order diffusion-wave equation by means of the theory of an abstract
Volterra equation \cite{atanackovic2009time}.
In \cite{Gorenflo2013}, for the one-dimensional distributed-order diffusion-wave equation,
R. Gorenflo et al. gave the interpretation of the fundamental solution of the Cauchy problem as a probability density function of the space variable $x$ evolving in time $t$ in the transform domain
by employing the technique of the Fourier and Laplace transforms.
Using the Laplace transform method, Z. Li et al investigated the asymptotic behavior
of solutions to the initial-boundary-value problem for the distributed-order time-fractional diffusion
equations \cite{LiLuchkoYamamoto2014}.

In most instances, the analytical solutions of distributed-order differential equations are not easy
to available, thus it stimulates researchers to develop numerical algorithms for approximate solutions.
To our knowledge, the research on numerically solving the distributed-order differential equations
are still in its infancy.
The literatures \cite{DiethelmFord2009}\cite{Podlubny2013}\cite{Katsikadelis2014}
concerned on developing numerical methods for solving distributed-order ordinary differential equations.
In terms of the distributed-order partial differential equations, most of the work are about the one-dimensional
time distributed-order differential equations,
and the integrating range of the order of time derivative is the interval $[0,1]$,
which is named as time distributed-order diffusion equation.
N. J. Ford et al. developed an implicit finite difference method for the solution of the diffusion
equation with distributed order in time \cite{Ford2015}. By using the Gr{\"{u}}nwald-Letnikov formula,
Gao et al. proposed two difference schemes to solve the one-dimensional distributed-order differential equations,
and the extrapolation method was applied to improve the approximate accuracy \cite{GaoSun2016111}.
In \cite{GaoSunSun2015}, the authors handled the same distributed-order differential equations by employing
a weighted and shifted Gr{\"{u}}nwald-Letnikov formula to derive several second-order convergent difference
schemes.
When the order of the time derivative is distributed over the interval $[1,2]$, it is called the time
distributed-order wave equation. The study of the numerical solution of this kind of equation
is rather more limited.
Ye et al. derived and analysed a compact
difference scheme for a distributed-order time-fractional wave equation in \cite{YeLiuAnh2015}.

When considering the high-dimensional models, Gao et al. investigated ADI schemes for two-dimensional
distributed-order diffusion equations \cite{GaoSun2015}\cite{GaoSun2016}, and they also developed two
ADI difference schemes for solving the two-dimensional time distributed-order wave equations \cite{GaoSun2016a}.
Due to the widespread use of the nonlinear models \cite{RidaEl-SayedArafa2010}\cite{WazwazGorguis2004},
M. L. Morgado et al. developed an implicit difference scheme for one-dimensional time distributed-order
diffusion equation with a nonlinear source term \cite{MorgadoRebelo2015}.
For further discussion on the numerical approaches for solving the high-dimensional distributed-order
partial differential equations,
this paper is devoted to develop effective numerical algorithm for two-dimensional time-fractional wave equation of distributed-order with a
nonlinear source term
\begin{align}\label{eq:equation}
&\int_1^2p(\beta){}_0^CD_t^\beta u(x,y,t)d\beta=\frac{\partial^2u(x,y,t)}{\partial x^2}+\frac{\partial^2u(x,y,t)}{\partial x^2}
+f\big(x,y,t,u(x,y,t)\big), \nonumber \\
&(x,y)\in \Omega,\quad t\in(0,T],   \\  \label{eq:boundaryy}
&u(x,y,t)=\phi(x,y,t),\quad (x,y)\in\partial\Omega,\quad 0\leq t<T,\\  \label{eq:initiall}
&u(x,y,0)=\psi_1(x,y),\quad u_t(x,y,0)=\psi_2(x,y),\quad (x,y)\in \Omega,
\end{align}
where $\Omega=(0,L_1)\times(0,L_2)$, and $\partial \Omega$ is the boundary of $\Omega$.
The fractional derivative ${}_0^CD_t^\beta v(t)$ in~\eqref{eq:equation} is given in the Caputo sense
\begin{equation*}
{}_0^CD_t^\beta v(t)=
\left\{
\begin{aligned}
&\frac{\partial v(t)}{\partial t}-\frac{\partial v(0)}{\partial t},\ \beta=1,\\
&\frac{1}{\Gamma{(2-\beta)}}\int_0^t(t-\xi)^{1-\beta}\frac{\partial^2 v(\xi)}{\partial \xi^2}d\xi,\ 1<\beta<2,\\
&\frac{\partial^2v(t)}{\partial t^2},\ \beta=2,
\end{aligned}\right.
\end{equation*}
and the function $p(\beta)$ is served as weight for the order of differentiation
such that $p(\beta)>0$ and $\int_1^2p(\beta)d\beta=c_0>0$.
We assume that $p(\beta)$, $\phi(x,y,t)$, $\psi_1(x,y)$, $\psi_2(x,y)$ and $f(x,y,t,u)$ are continuous,
and the nonlinear source term $f$ satisfies a Lipschitz condition of the form
\begin{equation}\label{eq:lip}
|f(x,y,t,u_1)-f(x,y,t,u_2)|\leq L_f|u_1-u_2|,
\end{equation}
where $L_{f}$ is a positive constant.

The main procedure of developing numerical scheme for solving problem \eqref{eq:equation}$-$\eqref{eq:initiall} is as
follows. Firstly a suitable numerical quadrature formula is adopted to discrete the integral
in \eqref{eq:equation}, and a multi-term time fractional wave equation is left whereafter.
Then we develop an ADI finite difference scheme which is uniquely solvable for the multi-term time
fractional wave equation. By using the discrete energy method, we prove the derived numerical scheme is unconditionally stable and convergent.

The rest of this paper is organized in the following way. In Section 2, the
ADI finite difference scheme is constructed and described detailedly.
In Section 3, we give analysis on solvability, stability and
convergence for the derived difference scheme. Numerical results are illustrated in Section 4 to confirm the
effectiveness and accuracy of our method, and some conclusions are drawn in the last section.


\section[The derivation of the ADI scheme]{The derivation of the ADI scheme}\label{sec:scheme}
This section focuses on deriving the ADI scheme for the
problems~\eqref{eq:equation}$-$\eqref{eq:initiall}.

Let ${M_1}$, ${M_2}$ and $N$ be positive integers,
and $h_1=L_1/M_1$, $h_2=L_2/M_2$ and $\tau=T/N$ be the uniform sizes of spatial grid and time step, respectively.
Then a spatial and temporal partition can be defined
as $x_i=ih_1$ for $i=0,1,\cdots,M_1$, $y_j=jh_2$ for $j=0,1,\cdots,M_2$ and $t_n=n\tau$ for $n=0,1,\cdots,N$.
Denote $\bar{\Omega}_h=\{(x_i,y_j)\mid 0\leq i\leq M_1,0\leq j\leq M_2\}$
and $\Omega_\tau=\{t_n\mid t_n=n\tau,0\leq n\leq N\}$,
then the domain $\bar{\Omega}\times [0,T]$ is covered by $\bar{\Omega}_h\times\Omega_\tau$.
Let $u=\{u_{ij}^n\mid 0\leq i\leq M_1,0\leq j\leq M_2,0\leq n\leq N \}$ be a grid function on $\bar{\Omega}_h\times\Omega_\tau$.
We introduce the following notations:
$$u_{ij}^{n-\frac{1}{2}}=\frac{1}{2}(u_{ij}^n+u_{ij}^{n-1}),\qquad
\delta_tu_{ij}^{n-\frac{1}{2}}=\frac{1}{\tau}(u_{ij}^n-u_{ij}^{n-1}),$$
$$\delta_xu_{i-\frac{1}{2},j}^n=\frac{1}{h_1}(u_{ij}^n-u_{i-1,j}^n),\qquad
\delta_x^2u_{ij}^n=\frac{1}{h_1}(\delta_xu_{i+\frac{1}{2},j}^n-\delta_xu_{i-\frac{1}{2},j}^n),$$
$$\delta_yu_{i,j-\frac{1}{2}}^n=\frac{1}{h_2}(u_{ij}^n-u_{i,j-1}^n),\qquad
\delta_y^2u_{ij}^n=\frac{1}{h_2}(\delta_xu_{i,j+\frac{1}{2}}^n-\delta_xu_{i,j-\frac{1}{2}}^n),$$
and
$$\Delta_hu_{ij}=\delta_x^2u_{ij}+\delta_y^2u_{ij}.$$

Consider Eq.~\eqref{eq:equation} at the point $(x_i,y_j,t_n)$, and we write it as
\begin{equation}\label{eq:point}
\begin{aligned}
&\int_1^2p(\beta){}_0^CD_t^{\beta}u(x_i,y_j,t_n)d\beta\\
=&\frac{\partial^2u(x_i,y_j,t_n)}{\partial x^2}
+\frac{\partial^2u(x_i,y_j,t_n)}{\partial y^2}+f\big(x_i,y_j,t_n,u(x_i,y_j,t_n)\big).
\end{aligned}
\end{equation}
Take an average of Eq. \eqref{eq:point} on time level $t=t_n$ and $t=t_{n-1}$, then we have
\begin{equation}\label{eq:average}
\begin{aligned}
&\frac{1}{2}\bigg(\int_1^2p(\beta){}_0^CD_t^\beta u(x_i,y_j,t_n)d\beta
+\int_1^2p(\beta){}_0^CD_t^\beta u(x_i,y_j,t_{n-1})d\beta\bigg)\\
=& \frac{1}{2}\bigg[\frac{\partial^2u(x_i,y_j,t_n)}{\partial x^2}
+\frac{\partial^2u(x_i,y_j,t_{n-1})}{\partial x^2}\bigg]
+\frac{1}{2}\bigg[\frac{\partial^2u(x_i,y_j,t_n)}{\partial y^2}
+\frac{\partial^2u(x_i,y_j,t_{n-1})}{\partial y^2}\bigg]\\
&+\frac{1}{2}\bigg[f\big(x_i,y_j,t_n,u(x_i,y_j,t_n)\big)+f\big(x_i,y_j,t_{n-1},u(x_i,y_j,t_{n-1})\big)\bigg].
\end{aligned}
\end{equation}
Denote by $U_{ij}^n=u(x_i,y_j,t_n)$
the grid functions on $\bar{\Omega}_h\times\Omega_\tau$ with $0\leq i\leq M_1$,
$0\leq j\leq M_2$, $0\leq n\leq N$. Eq. \eqref{eq:average} can be expressed as
\begin{equation}\label{eq:gridU}
\begin{aligned}
\int_1^2p(\beta){}_0^CD_t^\beta U_{ij}^{n-\frac{1}{2}}d\beta= & \frac{\partial^2}
{\partial x^2}U_{ij}^{n-\frac{1}{2}}+\frac{\partial^2}{\partial y^2}U_{ij}^{n-\frac{1}{2}}\\
&+\frac{1}{2}\bigg[f\big(x_i,y_j,t_n,U_{ij}^n\big)+f\big(x_i,y_j,t_{n-1},U_{ij}^{n-1}\big)\bigg]
\end{aligned}
\end{equation}

Firstly we discretize the integral term in~\eqref{eq:gridU}.
Suppose $p(\beta)\in C^2[1,2]$, ${}_0^CD_t^\beta u(x_i,y_j,t)|_{t=t_{n-1}}$ and
${}_0^CD_t^\beta u(x_i,y_j,t)|_{t=t_n}\in C^2[1,2]$.
Let $K$ be a positive integer, and
$\Delta\beta=1/K$ be the uniform step size. Take $\beta_l=1+\frac{2l-1}{2}\Delta\beta$, $1\leq l \leq K$,
then the mid-point quadrature rule is used for approximating the integral in~\eqref{eq:gridU}
\begin{equation}\label{eq:multii}
\begin{aligned}
&\Delta\beta\sum_{l=1}^{K}p(\beta_l){}_0^{C}D_t^{\beta_l}U_{ij}^{n-\frac{1}{2}}+R_1
=\frac{\partial^2}{\partial x^2}U_{ij}^{n-\frac{1}{2}}+\frac{\partial^2}{\partial y^2}U_{ij}^{n-\frac{1}{2}}\\
&+\frac{1}{2}\bigg[f\big(x_i,y_j,t_n,U(x_i,y_j,t_n)\big)+f\big(x_i,y_j,t_{n-1},U(x_i,y_j,t_{n-1})\big)\bigg],
\end{aligned}
\end{equation}
where $R_1=\mathcal{O}(\Delta\beta^2)$.

Next, we solve the multi-term time fractional wave equation~\eqref{eq:multii} with the initial and boundary
conditions~\eqref{eq:initiall} and~\eqref{eq:boundaryy}.
Suppose $u(x,y,t)\in C_{x,y,t}^{4,4,3}(\bar{\Omega}\times [0,T])$.
According to Theorem 8.2.5 in \cite{Sun2009},
the Caputo derivative ${}_0^CD_t^{\beta_l}U_{ij}^{n-\frac{1}{2}}$, $1 < \beta_l< 2$
have the fully discrete difference scheme
\begin{equation}\label{eq:fractional}
\begin{aligned}
&{}_0^CD_t^{\beta_l}U_{ij}^{n-\frac{1}{2}}\\
=&\frac{\tau^{1-\beta_l}}{\Gamma(3-\beta_l)}\bigg[a_0^{(\beta_l)}\delta_tU_{ij}^{n-\frac{1}{2}}
-\sum_{k=1}^{n-1}\big(a_{n-k-1}^{(\beta_l)}-
a_{n-k}^{(\beta_l)}\big)\delta_tU_{ij}^{k-\frac{1}{2}}-a_{n-1}^{(\beta_l)}\psi_2(x_i,y_j)\bigg]+R_2^l,
\end{aligned}
\end{equation}
where
\begin{equation*}
a_k^{(\beta_l)}=(k+1)^{2-\beta_l}-k^{2-\beta_l},\quad k=0,1,2,\cdots,
\end{equation*}
and
\begin{equation}\label{eq:lefterm}
\begin{aligned}
\mid R_2^l\mid\leq & \frac{1}{\Gamma(3-\beta_l)}\bigg[\frac{2-\beta_l}{12}+\frac{2^{3-\beta_l}}{3-\beta_l}
-(1+2^{1-\beta_l})+\frac{1}{12}\bigg]\cdot\\
& \max_{0\leq t\leq t_n}\mid\frac{\partial^3u(x_i,y_j,t)}{\partial t^3}\mid\tau^{3-\beta_l},\quad
l=1,2,\cdots,K.
\end{aligned}
\end{equation}
In the meantime, using the second order finite difference
\begin{equation*}
\frac{\partial^2g(x_i)}{\partial x^2}=\frac{g(x_{i+1})-2g(x_i)+g(x_{i-1})}{(\Delta x)^2}
-\frac{(\Delta x)^2}{12}\frac{\partial^4g(\xi_i)}{\partial x^4},\quad \xi_i \in (x_{i-1},x_{i+1})
\end{equation*}
to approximate the second order derivatives in~\eqref{eq:multii},
it is obtained
\begin{equation}\label{eq:leftnon}
\begin{aligned}
&\Delta\beta\sum_{l=1}^Kp(\beta_l)\frac{\tau^{1-\beta_l}}{\Gamma(3-\beta_l)}
\bigg[a_0^{(\beta_l)}\delta_tU_{ij}^{n-\frac{1}{2}}-\sum_{k=1}^{n-1}\big(a_{n-k-1}^{(\beta_l)}-
a_{n-k}^{(\beta_l)}\big)\delta_tU_{ij}^{k-\frac{1}{2}}\\
&-a_{n-1}^{(\beta_l)}\psi_2(x_i,y_j)\bigg]+\sum_{l=1}^K\Delta\beta p(\beta_l)R_2^l+R_1\\
=& \delta_x^2U_{ij}^{n-\frac{1}{2}}+\delta_y^2U_{ij}^{n-\frac{1}{2}}
+\frac{1}{2}\Big(f\big(x_i,y_j,t_{n-1},U_{ij}^{n-1}\big)+f\big(x_i,y_j,t_{n},U_{ij}^{n}\big)\Big)+R_3,
\end{aligned}
\end{equation}
where $R_3=\mathcal{O}(h_1^2+h_2^2)$.
Subsequently, the nonlinear source term is dealt with in the following manner to
avoid a system of nonlinear equations when computing:
\begin{equation}\label{eq:nondis}
f(x_i,y_j,t_n,U_{ij}^n)=f(x_i,y_j,t_{n-1},U_{ij}^{n-1})+\mathcal{O}(\tau).
\end{equation}
Substituting~\eqref{eq:nondis} in~\eqref{eq:leftnon}, we are left with
\begin{equation}\label{eq:midsch}
\begin{aligned}
&\Delta\beta\sum_{l=1}^Kp(\beta_l)\frac{\tau^{1-\beta_l}}{\Gamma(3-\beta_l)}
\bigg[a_0^{(\beta_l)}\delta_tU_{ij}^{n-\frac{1}{2}}-\sum_{k=1}^{n-1}\big(a_{n-k-1}^{(\beta_l)}-
a_{n-k}^{(\beta_l)}\big)\delta_tU_{ij}^{k-\frac{1}{2}}-a_{n-1}^{(\beta_l)}\psi_2(x_i,y_j)\bigg]\\
&=\delta_x^2U_{ij}^{n-\frac{1}{2}}+\delta_y^2U_{ij}^{n-\frac{1}{2}}
+f\big(x_i,y_j,t_{n-1},U_{ij}^{n-1}\big)+R_{ij}^{n-\frac{1}{2}}+\widetilde{R}_{ij}^{n-\frac{1}{2}},
\end{aligned}
\end{equation}
where
$$R_{ij}^{n-\frac{1}{2}}=-\sum_{l=1}^K\Delta\beta p(\beta_l)R_2^l+\mathcal{O}(h_1^2+h_2^2)
+\mathcal{O}(\Delta\beta^2)$$
and
$$\widetilde{R}_{ij}^{n-\frac{1}{2}}=\mathcal{O}(\tau).$$
From~\eqref{eq:lefterm}, we can deduce that there exists a positive constant $C_1$ such that
$$\bigg|-\sum_{l=1}^K\Delta\beta p(\beta_l)R_2^l \bigg| \leq C_1\tau^{1+\frac{1}{2}\Delta\beta}
\sum_{l=1}^K\Delta\beta p(\beta_l).$$
Since
$$\sum_{l=1}^K\Delta\beta p(\beta_l)\sim \int_1^2p(\beta)d\beta=c_0,$$
we get
$$\sum_{l=1}^K\Delta\beta p(\beta_l)\leq C_2,$$
where $C_2$ is a positive constant.
Thus there exists a positive constant $C_3$ such that
$$\Big|R_{ij}^{n-\frac{1}{2}}\Big|\leq C_3\left(\tau^{1+\frac{1}{2}\Delta\beta}
+h_1^2+h_2^2+\Delta\beta^2\right).$$
Besides, it is obvious that $$\Big|\widetilde{R}_{ij}^{n-\frac{1}{2}}\Big| \leq C_4\tau,$$
where $C_4$ is a positive constant.

Denote $$\mu=\Delta\beta\sum_{l=1}^Kp(\beta_l)\frac{1}{\tau^{\beta_l}\Gamma(3-\beta_l)}.$$
Since
\begin{equation}\nonumber
\begin{aligned}
&\Delta\beta\sum_{l=1}^Kp(\beta_l)\frac{1}{\tau^{\beta_l}\Gamma(3-\beta_l)}\\
\sim & \int_1^2p(\beta)\frac{1}{\tau^\beta\Gamma(3-\beta)}d\beta\\
=&\frac{p(\beta^\ast)}{\Gamma(3-\beta^\ast)}\int_1^2\frac{1}{\tau^\beta}d\beta\\
=&\frac{p(\beta^\ast)}{\Gamma(3-\beta^\ast)}\frac{1-\tau}{\tau^2\mid \ln\tau \mid},
\end{aligned}
\end{equation}
it can be concluded that
$$\mu=\frac{1}{\mathcal{O}(\tau^2| \ln\tau|)}.$$
In addition, $|\ln\tau| \leq C\tau^{-\varepsilon}$ for any positive and small $\varepsilon$
when $\tau$ is sufficiently small, thus the term $\mathcal{O}(\tau^2|\ln\tau |)$
is almost the same as $\mathcal{O}(\tau^2)$ when $\tau$ is sufficiently small.
Adding the high order term
$$\frac{\tau}{4\mu}\delta_x^2\delta_y^2\frac{U_{ij}^n-U_{ij}^{n-1}}{\tau}$$
on both sides of~\eqref{eq:midsch}, we derive
\begin{equation}\label{eq:U}
\begin{aligned}
&\Delta\beta\sum_{l=1}^Kp(\beta_l)\frac{\tau^{1-\beta_l}}{\Gamma(3-\beta_l)}
\bigg[a_0^{(\beta_l)}\delta_tU_{ij}^{n-\frac{1}{2}}-\sum_{k=1}^{n-1}\big(a_{n-k-1}^{(\beta_l)}-
a_{n-k}^{(\beta_l)}\big)\delta_tU_{ij}^{k-\frac{1}{2}}-a_{n-1}^{(\beta_l)}\psi_2(x_i,y_j)\bigg]\\
&+\frac{\tau}{4\mu}\delta_x^2\delta_y^2\frac{U_{ij}^n-U_{ij}^{n-1}}{\tau}\\
=&\delta_x^2U_{ij}^{n-\frac{1}{2}}+\delta_y^2U_{ij}^{n-\frac{1}{2}}
+f\big(x_i,y_j,t_{n-1},U_{ij}^{n-1}\big)+R_{ij}^{n-\frac{1}{2}}+\widetilde{R}_{ij}^{n-\frac{1}{2}}+\widehat{R}_{ij}^{n-\frac{1}{2}},
\end{aligned}
\end{equation}
where $$\widehat{R}_{ij}^{n-\frac{1}{2}}=\frac{\tau}{4\mu}\delta_x^2\delta_y^2\frac{U_{ij}^n-U_{ij}^{n-1}}{\tau},$$
and it is clear that
$$\Big|\widehat{R}_{ij}^{n-\frac{1}{2}}\Big|\leq C_5\tau^3|\ln\tau|.$$
Also, for the initial and boundary value conditions, we have
\begin{equation}
\begin{aligned}
&U_{ij}^0=\psi_1(x_i,y_j),\ (x_i,y_j)\in \Omega ,
\end{aligned}
\end{equation}
\begin{equation}\label{eq:Ubou}
\begin{aligned}
&U_{ij}^n=\phi(x_i,y_j,t_n),\ (x_i,y_j)\in \partial\Omega,\ 0\leq n\leq N.
\end{aligned}
\end{equation}

Let $u_{ij}^n$ be the numerical approximation to $u(x_i,y_j,t_n)$.
Neglecting the small term $R_{ij}^{n-\frac{1}{2}}$, $\widetilde{R}_{ij}^{n-\frac{1}{2}}$ and $\widehat{R}_{ij}^{n-\frac{1}{2}}$ in \eqref{eq:U}, and
using $u_{ij}^n$ instead of $U_{ij}^n$ in \eqref{eq:U}$-$\eqref{eq:Ubou}, we construct the difference scheme for
\eqref{eq:equation}$-$\eqref{eq:initiall} as follows:
\begin{align}\label{eq:scheme}
&\Delta\beta\sum_{l=1}^Kp(\beta_l)\frac{\tau^{1-\beta_l}}{\Gamma(3-\beta_l)}
\bigg[a_0^{(\beta_l)}\delta_tu_{ij}^{n-\frac{1}{2}}-\sum_{k=1}^{n-1}\big(a_{n-k-1}^{(\beta_l)}-
a_{n-k}^{(\beta_l)}\big)\delta_tu_{ij}^{k-\frac{1}{2}}\nonumber\\
&-a_{n-1}^{(\beta_l)}(\psi_2)_{ij}\bigg]\nonumber
+\frac{\tau}{4\mu}\delta_x^2\delta_y^2\frac{u_{ij}^n-u_{ij}^{n-1}}{\tau}\nonumber\\
=&\delta_x^2u_{ij}^{n-\frac{1}{2}}+\delta_y^2u_{ij}^{n-\frac{1}{2}}
+f\big(x_i,y_j,t_{n-1},u_{ij}^{n-1}\big),\nonumber\\
&1\leq i\leq M_1-1,\ 1\leq j\leq M_2-1,\ 1\leq n\leq N,\\ \label{eq:schini}
&u_{ij}^0=(\psi_1)_{ij},\   1 \leq i\leq M_1-1,\ 1\leq j\leq M_2-1,\\  \label{eq:schboun}
&u_{ij}^n=\phi_{ij}^n, \ (i,j)\in \gamma=\big\{(i,j)\ |\ (x_i,y_j)\in\partial\Omega\big\},\  0 \leq n\leq N,
\end{align}
where
$$(\psi_1)_{ij}=\psi_1(x_i,y_j),\ (\psi_2)_{ij}=\psi_2(x_i,y_j),\ 1 \leq i\leq M_1-1,\ 1\leq j\leq M_2-1,$$
and
$$\phi_{ij}^n=\phi(x_i,y_j,t_n),\quad (i,j)\in\gamma,\quad 0 \leq n\leq N.$$

Notice $a_0^{(\beta_l)}=1$, then Eq.~\eqref{eq:scheme} can be rewritten as:
\begin{equation*}\label{eq:add}
\begin{aligned}
&\Delta\beta\sum_{l=1}^Kp(\beta_l)\frac{1}{\tau^{\beta_l}\Gamma{(3-\beta_l)}}u_{ij}^n
-\frac{1}{2}\delta_x^2u_{ij}^n-\frac{1}{2}\delta_y^2u_{ij}^n
+\frac{1}{4\mu}\delta_x^2\delta_y^2u_{ij}^n\\
=&\Delta\beta\sum_{l=1}^Kp(\beta_l)\frac{1}{\tau^{\beta_l}\Gamma{(3-\beta_l)}}
\bigg[u_{ij}^{n-1}+\sum_{k=1}^{n-1}\big(a_{n-k-1}^{(\beta_l)}-a_{n-k}^{(\beta_l)}\big)\big(u_{ij}^k-u_{ij}^{k-1}\big)\\
&+\tau a_{n-1}^{(\beta_l)}(\psi_2)_{ij}\bigg]+\frac{1}{2}\delta_x^2u_{ij}^{n-1}+\frac{1}{2}\delta_y^2u_{ij}^{n-1}
+\frac{1}{4\mu}\delta_x^2\delta_y^2u_{ij}^{n-1}+f\big(x_i,y_j,t_{n-1},u_{ij}^{n-1}\big),
\end{aligned}
\end{equation*}
or
\begin{equation*}
\begin{aligned}
&\left(\sqrt{\mu}I-\frac{1}{2\sqrt{\mu}}\delta_x^2\right)\bigg(\sqrt{\mu}I-\frac{1}{2\sqrt{\mu}}\delta_y^2\bigg)u_{ij}^n\\
=&\left(\sqrt{\mu}I+\frac{1}{2\sqrt{\mu}}\delta_x^2\right)\bigg(\sqrt{\mu}I+\frac{1}{2\sqrt{\mu}}\delta_y^2\bigg)u_{ij}^{n-1}
+\Delta\beta\sum_{l=1}^Kp(\beta_l)\frac{1}{\tau^{\beta_l}\Gamma(3-\beta_l)}\cdot\\
&\bigg[\sum_{k=1}^{n-1}\big(a_{n-k-1}^{(\beta_l)}-a_{n-k}^{(\beta_l)}\big)\big(u_{ij}^k-u_{ij}^{k-1}\big)+\tau a_{n-1}^{(\beta_l)}(\psi_2)_{ij}\bigg]
+f\big(x_i,y_j,t_{n-1},u_{ij}^{n-1}\big),
\end{aligned}
\end{equation*}
where $I$ denotes the identity operator.

Let
$$u_{ij}^{\ast}=\left(\sqrt{\mu}I-\frac{1}{2\sqrt{\mu}}\delta_y^2\right)u_{ij}^n.$$
Together with \eqref{eq:schini} and \eqref{eq:schboun} the ADI difference scheme is derived, and the procedure can be executed as follows:

On each time level $t=t_n$ $(1\leq n\leq N)$, firstly, for all fixed $y=y_j$ $(1\leq j\leq M_2-1)$,
solving a set of $M_1-1$ equations at the mesh points $x_i$ $(1\leq i\leq M_1-1)$ to get the
intermediate solution $u_{ij}^{\ast}$:
\begin{equation}\label{eq:x}
\left\{
\begin{aligned}
&\left(\sqrt{\mu}I-\frac{1}{2\sqrt{\mu}}\delta_x^2\right)u_{ij}^\ast
=\left(\sqrt{\mu}I+\frac{1}{2\sqrt{\mu}}\delta_x^2\right)\bigg(\sqrt{\mu}I+\frac{1}{2\sqrt{\mu}}\delta_y^2\bigg)u_{ij}^{n-1}\\
&+\Delta\beta\sum_{l=1}^Kp(\beta_l)\frac{1}{\tau^{\beta_l}\Gamma(3-\beta_l)}
\bigg[\sum_{k=1}^{n-1}\big(a_{n-k-1}^{(\beta_l)}-a_{n-k}^{(\beta_l)}\big)\big(u_{ij}^k-u_{ij}^{k-1}\big)+\tau a_{n-1}^{(\beta_l)}(\psi_2)_{ij}\bigg]\\
&+f\big(x_i,y_j,t_{n-1},u_{ij}^{n-1}\big),\quad 1\leq i\leq M_1-1,\\
&u_{0j}^\ast=\bigg(\sqrt{\mu}I-\frac{1}{2\sqrt{\mu}}\delta_y^2\bigg)u_{0j}^n,\quad
u_{M_1j}^\ast=\bigg(\sqrt{\mu}I-\frac{1}{2\sqrt{\mu}}\delta_y^2\bigg)u_{M_1j}^n;
\end{aligned}\right.
\end{equation}
afterwards, for all fixed $x=x_i$ $(1\leq i\leq M_1-1)$, by computing a set of $M_2-1$ equations at the mesh points
$y_j$ $(1\leq j\leq M_2-1)$, the solution $u_{ij}^n$ can be obtained:
\begin{equation}\label{eq:y}
\left\{
\begin{aligned}
&\bigg(\sqrt{\mu}I-\frac{1}{2\sqrt{\mu}}\delta_y^2\bigg)u_{ij}^n=u_{ij}^\ast, \quad 1\leq j\leq M_2-1,\\
&u_{i0}^n=\phi(x_i,y_0,t_n),\quad u_{iM_2}^n=\phi(x_i,y_{M_2},t_n).
\end{aligned}\right.
\end{equation}

\section{Analysis of the ADI difference scheme}\label{sec:analysis}
\subsection{Solvability}
It is clear that the ADI scheme~\eqref{eq:x}$-$\eqref{eq:y} is a linear tridiagonal system in unknowns,
and the coefficient matrices are strictly diagonally dominant. Thus the scheme~\eqref{eq:x}$-$\eqref{eq:y} has a
unique solution. This result can be written as following.

\begin{theorem}
The ADI difference scheme~\eqref{eq:x}$-$\eqref{eq:y} is uniquely solvable.
\end{theorem}

\subsection{Stability}
In this subsection we prove the unconditional stability and the convergence of the difference scheme
\eqref{eq:x}$-$\eqref{eq:y}. We start with some auxiliary definitions and useful results.

Denote the space of grid functions on $\bar{\Omega}_h$
$$\mathcal{V}_h=\{v\mid v=\{v_{ij}\mid (x_i,y_j)\in \bar{\Omega}_h\}\ and\ v_{ij}=0\ if\ (x_i,y_j)\in
\partial\Omega_h\}.$$
For any grid function $v\in \mathcal{V}_h$, the following discrete norms and Sobolev seminorm are introduced:
\begin{equation*}
\|v\|=\sqrt{h_1h_2\sum_{i=1}^{M_1-1}\sum_{j=1}^{M_2-1}|v_{ij}|^2},\quad
\|\delta_x\delta_yv\|=\sqrt{h_1h_2\sum_{i=1}^{M_1}\sum_{j=1}^{M_2}|\delta_x\delta_yv_{i-\frac{1}{2},j-\frac{1}{2}}|^2},
\end{equation*}
\begin{equation*}
\|\delta_xv\|=\sqrt{h_1h_2\sum_{i=1}^{M_1}\sum_{j=1}^{M_2-1}|\delta_xv_{i-\frac{1}{2},j}|^2},\quad
\|\delta_yv\|=\sqrt{h_1h_2\sum_{i=1}^{M_1-1}\sum_{j=1}^{M_2}|\delta_yv_{i,j-\frac{1}{2}}|^2},
\end{equation*}
\begin{equation*}
\|\Delta_hv\|=\sqrt{h_1h_2\sum_{i=1}^{M_1-1}\sum_{j=1}^{M_2-1}|\Delta_hv_{ij}|^2},\quad
|v|_1=\sqrt{\|\delta_xv\|^2+\|\delta_yv\|^2}.
\end{equation*}
\begin{lemma}\label{lem:L2}
\cite{SamarskiiAndreev1976} For any grid function $v\in \mathcal{V}_h$, $\|v\|\leq \frac{1}{2\sqrt{3}}|v|_1$.
\end{lemma}
\begin{lemma}\label{lem:L22}
\cite{Sun2009} For any grid function $v\in \mathcal{V}_h$, $|v|_1\leq \frac{1}{2\sqrt{3}}\|\Delta_hv\|$.
\end{lemma}
\begin{lemma}\label{lem:126}
\cite{Sun2009} For any $G=\{G_1,G_2,G_3,\ldots\}$ and $q$, we have
\begin{equation*}
\begin{aligned}
&\sum_{n=1}^{m}\left[b_0G_n-\sum_{k=1}^{n-1}(b_{n-k-1}-b_{n-k})G_k-b_{n-1}q\right]G_n\\
\geq & \frac{t_m^{1-\alpha}}{2}\tau\sum_{n=1}^mG_n^2-\frac{t_m^{2-\alpha}}{2(2-\alpha)}q^2,\qquad m=1,2,3,\cdots,
\end{aligned}
\end{equation*}
where $$b_l=\frac{\tau^{2-\alpha}}{2-\alpha}[(l+1)^{2-\alpha}-l^{2-\alpha}],\qquad l=0,1,2,\cdots.$$
\end{lemma}

The discrete Gronwall's inequality is also introduced below since it is necessary to prove the stability and convergence of
the proposed method.
\begin{lemma}\label{lem:gronwall}
\cite{QuarteroniValli2008} Assume that $k_n$ and $p_n$ are nonnegative sequences, and the sequence $\Phi_n$ satisfies
$$\Phi_0\leq g_0,\qquad \Phi_n\leq g_0+\sum_{l=0}^{n-1}p_l+\sum_{l=0}^{n-1}k_l\Phi_l,\qquad n\geq 1,$$
where $g_0\geq 0$. Then the sequence $\Phi_n$ satisfies
$$\Phi_n\leq\left(g_0+\sum_{l=0}^{l-1}p_l\right)\exp\left(\sum_{l=0}^{n-1}k_l\right),\qquad n\geq 1.$$
\end{lemma}
Since the ADI difference scheme \eqref{eq:x}$-$\eqref{eq:y} is equivalent to \eqref{eq:scheme}$-$\eqref{eq:schboun} if the intermediate variable $u^\ast$ is eliminated,
we analyze the stability and convergence by employing the difference scheme \eqref{eq:scheme}$-$\eqref{eq:schboun}.

Assume that $\widetilde u_{ij}^n$ is the approximate solution of $u_{ij}^n$,
which is the exact solution of the scheme~\eqref{eq:scheme}$-$\eqref{eq:schboun}.
Denote $\varepsilon_{ij}^n=u_{ij}^n-\widetilde u_{ij}^n,\ 0 \leq i\leq M_1,\ 0 \leq j\leq M_2, \ 0 \leq n\leq N$,
then we have the perturbation error equations
\begin{align}\label{eq:pertt}
&\Delta\beta\sum_{l=1}^Kp(\beta_l)\frac{\tau^{1-\beta_l}}{\Gamma(3-\beta_l)}
\bigg[a_0^{(\beta_l)}\delta_t\varepsilon_{ij}^{n-\frac{1}{2}}-\sum_{k=1}^{n-1}\big(a_{n-k-1}^{(\beta_l)}-
a_{n-k}^{(\beta_l)}\big)\delta_t\varepsilon_{ij}^{k-\frac{1}{2}}-a_{n-1}^{(\beta_l)}(\psi_2^\ast)_{ij}\bigg]\nonumber\\
&+\frac{\tau}{4\mu}\delta_x^2\delta_y^2\frac{\varepsilon_{ij}^n-\varepsilon_{ij}^{n-1}}{\tau}\nonumber\\
&=\delta_x^2\varepsilon_{ij}^{n-\frac{1}{2}}+\delta_y^2\varepsilon_{ij}^{n-\frac{1}{2}}
+f\big(x_i,y_j,t_{n-1},u_{ij}^{n-1}\big)-f\big(x_i,y_j,t_{n-1},\widetilde u_{ij}^{n-1}\big),\nonumber\\
&1\leq i\leq M_1-1,\ 1\leq j\leq M_2-1,\ 1\leq n\leq N,\\
&\varepsilon_{ij}^0=(\psi_1)_{ij}-(\widetilde\psi_1)_{ij},\   1 \leq i\leq M_1-1,\ 1\leq j\leq M_2-1,\nonumber\\
&\varepsilon_{ij}^n=0, \ (i,j)\in \gamma,\  0 \leq n\leq N \nonumber,
\end{align}
where $$(\psi_2^\ast)_{ij}=(\psi_2)_{ij}-(\widetilde\psi_2)_{ij}.$$

\begin{theorem}\label{them:stability}
Assume that the condition \eqref{eq:lip} is satisfied,
then the difference scheme \eqref{eq:x}$-$\eqref{eq:y} is unconditionally stable.
\end{theorem}
\begin{proof}
Let $$b_k^{(\beta_l)}=\frac{\tau^{2-{\beta_l}}}{2-{\beta_l}}a_k^{(\beta_l)},\quad 1\leq l\leq K,$$
then Eq. \eqref{eq:pertt} is equivalent to
\begin{equation}\label{eq:perterror}
\begin{aligned}
&\Delta\beta\sum_{l=1}^Kp(\beta_l)\frac{1}{\Gamma(2-\beta_l)\tau}
\bigg[b_0^{(\beta_l)}\delta_t\varepsilon_{ij}^{n-\frac{1}{2}}-\sum_{k=1}^{n-1}\big(b_{n-k-1}^{\beta_l}-
b_{n-k}^{(\beta_l)}\big)\delta_t\varepsilon_{ij}^{k-\frac{1}{2}}-b_{n-1}^{(\beta_l)}(\psi_2^\ast)_{ij}\bigg]\\
&+\frac{\tau}{4\mu}\delta_x^2\delta_y^2\frac{\varepsilon_{ij}^n-\varepsilon_{ij}^{n-1}}{\tau}\\
&=\delta_x^2\varepsilon_{ij}^{n-\frac{1}{2}}+\delta_y^2\varepsilon_{ij}^{n-\frac{1}{2}}
+f\big(x_i,y_j,t_{n-1},u_{ij}^{n-1}\big)-f\big(x_i,y_j,t_{n-1},\widetilde u_{ij}^{n-1}\big),\\
&1\leq i\leq M_1-1,\ 1\leq j\leq M_2-1,\ 1\leq n\leq N.
\end{aligned}
\end{equation}

Multiplying \eqref{eq:perterror} by $h_1h_2\tau\delta_t\varepsilon_{ij}^{n-\frac{1}{2}}$, summing up for $i$ from
$1$ to $M_1-1$, for $j$ from $1$ to $M_2-1$ and for $n$ from $1$ to $m$, we analyze each term
in the derived equation.
Firstly, by employing Lemma~\ref{lem:126}, we have
\begin{equation}\label{eq:left1}
\begin{aligned}
&\Delta\beta\sum_{l=1}^Kp(\beta_l)\frac{1}{\Gamma(2-\beta_l)}h_1h_2\sum_{i=1}^{M_1-1}\sum_{j=1}^{M_2-1}
\bigg\{\sum_{n=1}^m\Big[b_0^{(\beta_l)}\delta_t\varepsilon_{ij}^{n-\frac{1}{2}}-\\
&\sum_{k=1}^{n-1}\big(b_{n-k-1}^{(\beta_l)}-b_{n-k}^{(\beta_l)}\big)\delta_t\varepsilon_{ij}^{k-\frac{1}{2}}
-b_{n-1}^{(\beta_l)}(\psi_2^\ast)_{ij}\Big]
\delta_t\varepsilon_{ij}^{n-\frac{1}{2}}\bigg\}\\
\geq & \Delta\beta\sum_{l=1}^Kp(\beta_l)\frac{1}{\Gamma(2-\beta_l)}\Big[\frac{1}{2}t_m^{1-\beta_l}\tau
\sum_{n=1}^m\big\|\delta_t\varepsilon^{n-\frac{1}{2}}\big\|^2\\
&-\frac{t_m^{2-\beta_l}}{2(2-\beta_l)}h_1h_2\sum_{i=1}^{M_1-1}\sum_{j=1}^{M_2-1}(\psi_2^\ast)_{ij}^2\Big]\\
=&\frac{1}{2}\tau K_m\sum_{n=1}^m\big\|\delta_t\varepsilon^{n-\frac{1}{2}}\big\|^2
-\Delta\beta\sum_{l=1}^Kp(\beta_l)\frac{t_m^{2-\beta_l}}{2\Gamma(3-\beta_l)}\big\|\psi_2^\ast\big\|^2,
\end{aligned}
\end{equation}
where
$$K_m=\Delta\beta\sum_{l=1}^Kp(\beta_l)\frac{t_m^{1-\beta_l}}{\Gamma(2-\beta_l)}>0.$$
Whereafter using the discrete Green formula, we get
\begin{equation}
\begin{aligned}
&h_1h_2\tau\sum_{i=1}^{M_1-1}\sum_{j=1}^{M_2-1}\sum_{n=1}^m\frac{\tau}{4\mu}\delta_x^2\delta_y^2
\frac{\varepsilon_{ij}^{n}-\varepsilon_{ij}^{n-1}}{\tau}\delta_t\varepsilon_{ij}^{n-\frac{1}{2}}\\
=&\frac{1}{4\mu}\sum_{n=1}^mh_1h_2\sum_{i=1}^{M_1}\sum_{j=1}^{M_2}\bigg(\delta_x\delta_y
\Big(\varepsilon_{i-\frac{1}{2},j-\frac{1}{2}}^n-\varepsilon_{i-\frac{1}{2},j-\frac{1}{2}}^{n-1}\Big)\bigg)
\bigg(\delta_x\delta_y\Big(\varepsilon_{i-\frac{1}{2},j-\frac{1}{2}}^n-\varepsilon_{i-\frac{1}{2},j-\frac{1}{2}}^{n-1}\Big)\bigg)\\
=&\frac{1}{4\mu}\sum_{n=1}^m\big\|\delta_x\delta_y(\varepsilon^n-\varepsilon^{n-1})\big\|^2\geq0,
\end{aligned}
\end{equation}
and
\begin{equation}\label{eq:deltax}
\begin{aligned}
&\tau\sum_{n=1}^m\bigg[h_1h_2\sum_{i=1}^{M_1-1}\sum_{j=1}^{M_2-1}\Big(\delta_t\varepsilon_{ij}^{n-\frac{1}{2}}\Big)
\Big(\delta_x^2\varepsilon_{ij}^{n-\frac{1}{2}}\Big)\bigg]\\
=&-\tau\sum_{n=1}^m\bigg[h_1h_2\sum_{j=1}^{M_2-1}\sum_{i=1}^{M_1}\Big(\delta_x\varepsilon_{i-\frac{1}{2},j}^{n-\frac{1}{2}}\Big)
\Big(\delta_t\delta_x\varepsilon_{i-\frac{1}{2},j}^{n-\frac{1}{2}}\Big)\bigg]\\
=&-\tau\sum_{n=1}^m\bigg[h_1h_2\sum_{j=1}^{M_2-1}\sum_{i=1}^{M_1}\bigg(\frac{\delta_x\varepsilon_{i-\frac{1}{2},j}^n
+\delta_x\varepsilon_{i-\frac{1}{2},j}^{n-1}}{2}\bigg)\bigg(\frac{\delta_x\varepsilon_{i-\frac{1}{2},j}^n
-\delta_x\varepsilon_{i-\frac{1}{2},j}^{n-1}}{\tau}\bigg)\bigg]\\
=&-\frac{1}{2}\Big[\big\|\delta_x\varepsilon^m\big\|^2-\big\|\delta_x\varepsilon^0\big\|^2\Big].
\end{aligned}
\end{equation}
Analogous to \eqref{eq:deltax}, it is also obtained
\begin{equation}
\begin{aligned}
&\tau\sum_{n=1}^m\bigg[h_1h_2\sum_{i=1}^{M_1-1}\sum_{j=1}^{M_2-1}\Big(\delta_t\varepsilon_{ij}^{n-\frac{1}{2}}\Big)
\Big(\delta_y^2\varepsilon_{ij}^{n-\frac{1}{2}}\Big)\bigg]\\
=&-\frac{1}{2}\Big[\big\|\delta_y\varepsilon^m\big\|^2-\big\|\delta_y\varepsilon^0\big\|^2\Big].
\end{aligned}
\end{equation}
On the basis of \eqref{eq:lip}, there holds that
\begin{equation}\label{eq:right3}
\begin{aligned}
&h_1h_2\sum_{i=1}^{M_1-1}\sum_{j=1}^{M_2-1}\bigg[\tau\sum_{n=1}^m\Big(\delta_t\varepsilon_{ij}^{n-\frac{1}{2}}\Big)
\Big|f\Big(x_i,y_j,t_{n-1},u_{ij}^{n-1}\Big)-f\Big(x_i,y_j,t_{n-1},\widetilde{u}_{ij}^{n-1}\Big)\Big|\bigg]\\
\leq & h_1h_2\sum_{i=1}^{M_1-1}\sum_{j=1}^{M_2-1}\bigg[\tau\sum_{n=1}^m\Big(\delta_t\varepsilon_{ij}^{n-\frac{1}{2}}\Big)
L_f\Big|u_{ij}^{n-1}-\widetilde{u}_{ij}^{n-1}\Big|\bigg]\\
\leq & L_fh_1h_2\sum_{i=1}^{M_1-1}\sum_{j=1}^{M_2-1}\bigg[\tau\sum_{n=1}^m\Big(\delta_t\varepsilon_{ij}^{n-\frac{1}{2}}\Big)
\Big|\varepsilon_{ij}^{n-1}\Big|\bigg]\\
\leq & L_fh_1h_2\sum_{i=1}^{M_1-1}\sum_{j=1}^{M_2-1}\tau\sum_{n=1}^m
\bigg[\frac{K_m}{2L_f}\Big(\delta_t\varepsilon_{ij}^{n-\frac{1}{2}}\Big)^2
+\frac{L_f}{2K_m}\Big(\varepsilon_{ij}^{n-1}\Big)^2\bigg]\\
=&\frac{\tau K_m}{2}\sum_{n=1}^m\big\|\delta_t\varepsilon^{n-\frac{1}{2}}\big\|^2
+\frac{\tau L_f^2}{2K_m}\sum_{n=1}^m\big\|\varepsilon^{n-1}\big\|^2.
\end{aligned}
\end{equation}

From Equations \eqref{eq:left1}$-$\eqref{eq:right3}, the inequality below is derived
\begin{equation}\label{eq:inequ}
\begin{aligned}
&\big\|\delta_x\varepsilon^m\big\|^2+\big\|\delta_y\varepsilon^m\big\|^2\leq \big\|\delta_x\varepsilon^0\big\|^2+\big\|\delta_y\varepsilon^0\big\|^2\\
+& \Delta\beta\sum_{l=1}^Kp(\beta_l)\frac{t_m^{2-\beta_l}}{\Gamma(3-\beta_l)}\big\|\psi_2^\ast\big\|^2
+\frac{\tau L_f^2}{K_m}\sum_{n=1}^m\big\|\varepsilon^{n-1}\big\|^2.
\end{aligned}
\end{equation}
According to Lemma \ref{lem:L2} and Lemma \ref{lem:L22}, we deduce from \eqref{eq:inequ} that
\begin{equation*}
\begin{aligned}
\|\varepsilon^n\|^2\leq & \frac{1}{144}\|\Delta_h\varepsilon^0\|^2+\frac{1}{12}\Delta\beta\sum_{l=1}^Kp(\beta_l)
\frac{T^{2-\beta_l}}{\Gamma(3-\beta_l)}\|\psi_2^\ast\|^2\\
+& \frac{\tau L_f^2}{12\Delta\beta\sum_{l=1}^Kp(\beta_l)\frac{T^{1-\beta_l}}{\Gamma(2-\beta_l)}}\sum_{k=1}^n\|\varepsilon^{k-1}\|^2,
\qquad 1\leq n\leq N.
\end{aligned}
\end{equation*}
Finally, taking Lemma \ref{lem:gronwall}, it follows that
\begin{equation*}
\begin{aligned}
\|\varepsilon^n\|^2\leq & \bigg(\frac{1}{144}\|\Delta_h\varepsilon^0\|^2+\frac{1}{12}\Delta\beta\sum_{l=1}^Kp(\beta_l)
\frac{T^{2-\beta_l}}{\Gamma(3-\beta_l)}\|\psi_2^\ast\|^2\bigg)\cdot\\
&\exp{\Bigg(\frac{L_f^2}{12\Delta\beta\sum_{l=1}^Kp(\beta_l)\frac{T^{-\beta_l}}{\Gamma(2-\beta_l)}}}\Bigg).
\end{aligned}
\end{equation*}
This completes the proof.
\end{proof}

In the following we consider the convergence of the difference approximation. Noticing that $U_{ij}^n$ is the exact solution
of the system \eqref{eq:equation}$-$\eqref{eq:initiall} and $u_{ij}^n$ is the numerical solution of the difference
scheme \eqref{eq:scheme}$-$\eqref{eq:schboun}, we denote the error
$$e_{ij}^n=U_{ij}^n-u_{ij}^n,\quad 0\leq i\leq M_1,\quad 0\leq j\leq M_2,\quad 0\leq n\leq N.$$
Subscribing \eqref{eq:scheme}$-$\eqref{eq:schboun} from \eqref{eq:U}$-$\eqref{eq:Ubou}, we get the error equations
\begin{align}\label{eq:error}
&\Delta\beta\sum_{l=1}^Kp(\beta_l)\frac{\tau^{1-\beta_l}}{\Gamma(3-\beta_l)}\Big[a_0^{(\beta_l)}\delta_te_{ij}^{n-\frac{1}{2}}
-\sum_{k=1}^{n-1}\big(a_{n-k-1}^{(\beta_l)}-a_{n-k}^{(\beta_l)}\big)\delta_te_{ij}^{n-\frac{1}{2}}\Big]\nonumber\\
&+\frac{\tau}{4\mu}\delta_x^2\delta_y^2\frac{e_{ij}^n-e_{ij}^{n-1}}{\tau}\nonumber\\
= & \delta_x^2e_{ij}^{n-\frac{1}{2}}+\delta_y^2e_{ij}^{n-\frac{1}{2}}+f\big(x_i,y_j,t_{n-1},U_{ij}^{n-1}\big)
-f\big(x_i,y_j,t_{n-1},u_{ij}^{n-1}\big)\nonumber\\
&+R_{ij}^{n-\frac{1}{2}}+\widehat{R}_{ij}^{n-\frac{1}{2}}+\widetilde{R}_{ij}^{n-\frac{1}{2}},
\ 1\leq i\leq M_1-1,\ 1\leq j\leq M_2-1,\ 1\leq n\leq N,\\
&e_{ij}^0=0,\quad 1\leq i\leq M_1-1,\quad 1\leq j\leq M_2-1,\nonumber\\
&e_{ij}^n=0,\quad (i,j)\in \gamma,\quad 0\leq n\leq N.\nonumber
\end{align}
\begin{theorem}
Suppose that the continuous problem \eqref{eq:equation}$-$\eqref{eq:initiall} has
solution $u(x,y,t)\in C_{x,y,t}^{4,4,3}(\bar{\Omega}\times [0,T])$. Then there is a positive constant $C$ such that
\begin{equation}\nonumber
\big\|e^n\big\|\leq C (\tau+h_1^2+h_2^2+\Delta\beta^2).
\end{equation}
\end{theorem}
\begin{proof}
The proof of convergence is similar to that of Theorem \ref{them:stability}.
Multiplying \eqref{eq:error} by $h_1h_2\tau\delta_t e_{ij}^{n-\frac{1}{2}}$, summing up for $i$ from
$1$ to $M_1-1$, for $j$ from $1$ to $M_2-1$ and for $n$ from $1$ to $m$, we estimate each term in
the resulted equation.

By using analogous strategies as \eqref{eq:left1}$-$\eqref{eq:right3}, we get \eqref{eq:convleft}$-$\eqref{eq:convright1} correspondingly.
\begin{equation}\label{eq:convleft}
\begin{aligned}
&\Delta\beta\sum_{l=1}^Kp(\beta_l)\frac{\tau^{2-\beta_l}}{\Gamma(3-\beta_l)}h_1h_2\sum_{i=1}^{M_1-1}\sum_{j=1}^{M_2-1}
\bigg\{\sum_{n=1}^m\bigg[a_0^{(\beta_l)}\delta_te_{ij}^{n-\frac{1}{2}}-\\
&\sum_{k=1}^{n-1}\Big(a_{n-k-1}^{(\beta_l)}-a_{n-k}^{(\beta_l)}\Big)\delta_te_{ij}^{k-\frac{1}{2}}\bigg]
\delta_te_{ij}^{n-\frac{1}{2}}\bigg\}\\
\geq &\frac{1}{2}\tau K_m\sum_{n=1}^m\Big\|\delta_te^{n-\frac{1}{2}}\Big\|^2,
\end{aligned}
\end{equation}
\begin{equation}
\begin{aligned}
&h_1h_2\tau\sum_{i=1}^{M_1-1}\sum_{j=1}^{M_2-1}\sum_{n=1}^m\frac{\tau}{4\mu}\delta_x^2\delta_y^2
\frac{e_{ij}^{n}-e_{ij}^{n-1}}{\tau}\delta_te_{ij}^{n-\frac{1}{2}}\\
=&\frac{1}{4\mu}\sum_{n=1}^m\big\|\delta_x\delta_y(e^n-e^{n-1})\big\|^2\geq0,
\end{aligned}
\end{equation}
\begin{equation}
\begin{aligned}
\tau\sum_{n=1}^m\Big[h_1h_2\sum_{i=1}^{M_1-1}\sum_{j=1}^{M_2-1}\Big(\delta_te_{ij}^{n-\frac{1}{2}}\Big)
\Big(\delta_x^2e_{ij}^{n-\frac{1}{2}}\Big)\Big]
=-\frac{1}{2}\big\|\delta_xe^m\big\|^2,
\end{aligned}
\end{equation}
\begin{equation}
\begin{aligned}
\tau\sum_{n=1}^m\Big[h_1h_2\sum_{i=1}^{M_1-1}\sum_{j=1}^{M_2-1}\Big(\delta_te_{ij}^{n-\frac{1}{2}}\Big)
\Big(\delta_y^2e_{ij}^{n-\frac{1}{2}}\Big)\Big]
=-\frac{1}{2}\big\|\delta_ye^m\big\|^2,
\end{aligned}
\end{equation}
and
\begin{equation}\label{eq:convright1}
\begin{aligned}
&h_1h_2\sum_{i=1}^{M_1-1}\sum_{j=1}^{M_2-1}\bigg[\tau\sum_{n=1}^m\Big(\delta_te_{ij}^{n-\frac{1}{2}}\Big)
\Big|f\Big(x_i,y_j,t_{n-1},U_{ij}^{n-1}\Big)-f\Big(x_i,y_j,t_{n-1},u_{ij}^{n-1}\Big)\Big|\bigg]\\
\leq & L_fh_1h_2\sum_{i=1}^{M_1-1}\sum_{j=1}^{M_2-1}\bigg[\tau\sum_{n=1}^m\Big(\delta_te_{ij}^{n-\frac{1}{2}}\Big)
\Big|e_{ij}^{n-1}\Big|\bigg]\\
\leq & L_fh_1h_2\sum_{i=1}^{M_1-1}\sum_{j=1}^{M_2-1}\tau\sum_{n=1}^m
\bigg[\frac{K_m}{4L_f}\Big(\delta_te_{ij}^{n-\frac{1}{2}}\Big)^2
+\frac{L_f}{K_m}\Big(e_{ij}^{n-1}\Big)^2\bigg]\\
=&\frac{\tau K_m}{4}\sum_{n=1}^m\big\|\delta_te^{n-\frac{1}{2}}\big\|^2
+\frac{\tau L_f^2}{K_m}\sum_{n=1}^m\big\|e^{n-1}\big\|^2.
\end{aligned}
\end{equation}
As for the remainder, it is deduced that
\begin{equation}\label{eq:convright}
\begin{aligned}
&h_1h_2\sum_{i=1}^{M_1-1}\sum_{j=1}^{M_2-1}\sum_{n=1}^m\tau\big(\delta_te_{ij}^{n-\frac{1}{2}}\big)
\big(R_{ij}^{n-\frac{1}{2}}+\widetilde{R}_{ij}^{n-\frac{1}{2}}+\widehat{R}_{ij}^{n-\frac{1}{2}}\big)\\
\leq & h_1h_2\sum_{i=1}^{M_1-1}\sum_{j=1}^{M_2-1}\sum_{n=1}^m\tau\Big(\frac{K_m}{4}\big(\delta_te_{ij}^{n-\frac{1}{2}}\big)^2
+\frac{1}{K_m}\big(R_{ij}^{n-\frac{1}{2}}+\widetilde{R}_{ij}^{n-\frac{1}{2}}+\widehat{R}_{ij}^{n-\frac{1}{2}}\big)^2\Big)\\
\leq & \frac{\tau K_m}{4}\sum_{n=1}^m\big\|\delta_te^{n-\frac{1}{2}}\big\|^2+\frac{\tau h_1h_2}{K_m}\sum_{i=1}^{M_1-1}\sum_{j=1}^{M_2-1}\sum_{n=1}^m
\Big[C_3\big(\tau^{1+\frac{\Delta\beta}{2}}+h_1^2+h_2^2+\Delta\beta^2\big)\\
& +C_4\tau^3\big|\ln\tau\big|+C_5\tau\Big]^2\\
\leq & \frac{\tau K_m}{4}\sum_{n=1}^m\big\|\delta_te^{n-\frac{1}{2}}\big\|^2+\frac{\tau h_1h_2}{K_m}\sum_{i=1}^{M_1-1}\sum_{j=1}^{M_2-1}\sum_{n=1}^m
\Big[\big(C_3+C_4+C_5\big)\big(\tau+h_1^2+h_2^2+\Delta\beta^2\big)\Big]^2\\
\leq & \frac{\tau K_m}{4}\sum_{n=1}^m\big\|\delta_te^{n-\frac{1}{2}}\big\|^2
+\frac{TL_1L_2}{K_m}\Big[\big(C_3+C_4+C_5\big)\big(\tau+h_1^2+h_2^2+\Delta\beta^2\big)\Big]^2.
\end{aligned}
\end{equation}
From \eqref{eq:convleft}$-$\eqref{eq:convright} it follows that
\begin{equation*}
\begin{aligned}
\frac{1}{2}\big(\big\|\delta_xe^m\big\|^2+\big\|\delta_ye^m\big\|^2\big)
\leq & \frac{TL_1L_2}{K_m}\Big[(C_3+C_4+C_5\big)\big(\tau+h_1^2+h_2^2+\Delta\beta^2\big)\Big]^2\\
& +\frac{\tau L_f^2}{K_m}\sum_{n=1}^m\big\|e^{n-1}\big\|^2,
\end{aligned}
\end{equation*}
i.e.,
\begin{equation*}
\begin{aligned}
\frac{1}{2}\big|e^m\big|_1^2\leq & \frac{TL_1L_2}{\Delta\beta\sum_{l=1}^Kp(\beta_l)\frac{1}{\Gamma(2-\beta_l)}T^{1-\beta_l}}
\Big[(C_3+C_4+C_5\big)\big(\tau+h_1^2+h_2^2+\Delta\beta^2\big)\Big]^2\\
& +\frac{\tau L_f^2}{K_m}\sum_{n=1}^m\big\|e^{n-1}\big\|^2.
\end{aligned}
\end{equation*}
According to Lemma \ref{lem:L2}, we obtain
\begin{equation*}
\begin{aligned}
\big\|e^n\big\|^2\leq & \frac{TL_1L_2}{6\Delta\beta\sum_{l=1}^Kp(\beta_l)\frac{1}{\Gamma(2-\beta_l)}T^{1-\beta_l}}
\Big[(C_3+C_4+C_5\big)\big(\tau+h_1^2+h_2^2+\Delta\beta^2\big)\Big]^2\\
& +\frac{\tau L_f^2}{6K_n}\sum_{k=1}^n\big\|e^{k-1}\big\|^2,\quad 0\leq n\leq N.
\end{aligned}
\end{equation*}
Therefore,
\begin{equation*}
\begin{aligned}
\big\|e^n\big\|^2\leq & \frac{TL_1L_2}{6\Delta\beta\sum_{l=1}^Kp(\beta_l)\frac{1}{\Gamma(2-\beta_l)}T^{1-\beta_l}}
\Big[(C_3+C_4+C_5\big)\big(\tau+h_1^2+h_2^2+\Delta\beta^2\big)\Big]^2\cdot\\
& \exp{\bigg(\frac{L_f^2}{6\Delta\beta\sum_{l=1}^Kp(\beta_l)\frac{1}{\Gamma(2-\beta_l)}T^{-\beta_l}}}\bigg),
\end{aligned}
\end{equation*}
where Lemma \ref{lem:gronwall} is applied.
This completes the proof.
\end{proof}

\section{Numerical results}\label{sec:examples}
In this section, a numerical example is tested
to demonstrate the
effectiveness of the proposed scheme, and verify the theoretical results including convergence orders and
numerical stability.
The discrete $L^2$ and $L^\infty$ norms are both taken to measure the numerical errors.
Denote
$$\big\|e^N\big\|_{L^2}:=\Bigg(\sum_{j=1}^{M_2-1}\sum_{i=1}^{M_1-1}\big|U_{ij}^N-u_{ij}^N\big|^2h_1h_2\Bigg)^{\frac{1}{2}},$$
and
$$\big\|e^N\big\|_{L^\infty}:=\max_{1\leq j\leq M_2-1, 1\leq i\leq M_1-1}\big|U_{ij}^N-u_{ij}^N\big|.$$

\begin{example}\label{exam:a}
\begin{equation}
\begin{aligned}
&\int_1^2\Gamma(4-\beta){}_0^CD_t^{\beta}u(x,y,t)d\beta= \frac{\partial^2u(x,y,t)}{\partial x^2}
+\frac{\partial^2u(x,y,t)}{\partial y^2}\\
&+\sin x\sin y\bigg[2(t^3+2t+4)+\frac{6t^2-6t}{\ln t}\bigg]
-(t^3+2t+4)^2\sin^2 x\sin^2 y+u^2, \\
&0<t<1/2,\quad (x,y)\in \Omega=(0,\pi)\times (0,\pi),\\
&u(x,y,t)=0,\quad (x,y)\in \partial\Omega,\quad 0<t<1/2,\\
&u(x,y,0)=4\sin x\sin y,\quad u_t(x,y,0)=2\sin x\sin y, \quad (x,y)\in\Omega,
\end{aligned}
\end{equation}
whose analytical solution is known and is given by
$$u(x,y,t)=(t^3+2t+4)\sin x\sin y.$$
\end{example}

In Figure \ref{fig:error} we illustrate the relative error, which verifies the convergence of the algorithm we proposed.

In Figure \ref{fig:solution} we present a comparison of the exact and numerical solutions. It can be seen
that the numerical solution is in good agreement with the exact solution.

\begin{figure}[!htb]
\centering
\begin{minipage}[c]{0.6\textwidth}
  \centering
  \includegraphics[width=1.2\linewidth]{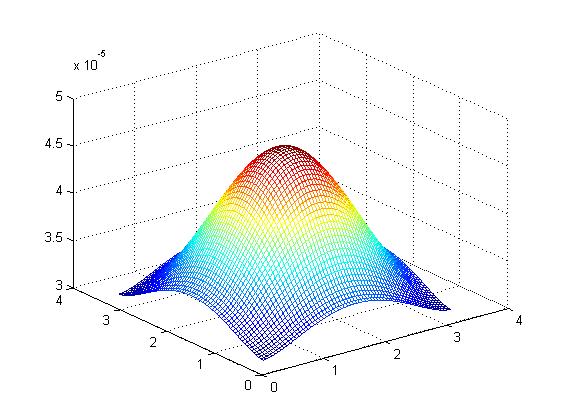} \\  
\end{minipage}
\caption{Relative error at $T=0.5$, obtained by algorithm \eqref{eq:x}$-$\eqref{eq:y} with mesh $h_1=h_2=\frac{\pi}{64}$, $\Delta\beta=\frac{1}{64}$, and $\tau=\frac{1}{4096}$.  }\label{fig:error}
\end{figure}

\begin{figure}[!htb]
\centering
\begin{minipage}[c]{0.45\textwidth}
  \centering
  \includegraphics[width=1\linewidth,totalheight=2in]{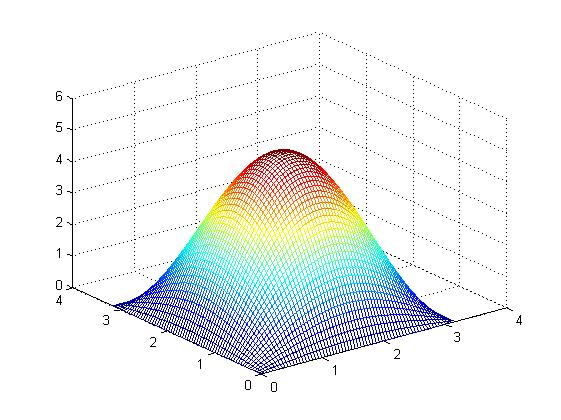} \\
  (a)
\end{minipage}
\begin{minipage}[c]{0.45\textwidth}
  \centering
  \includegraphics[width=1\linewidth,totalheight=2in]{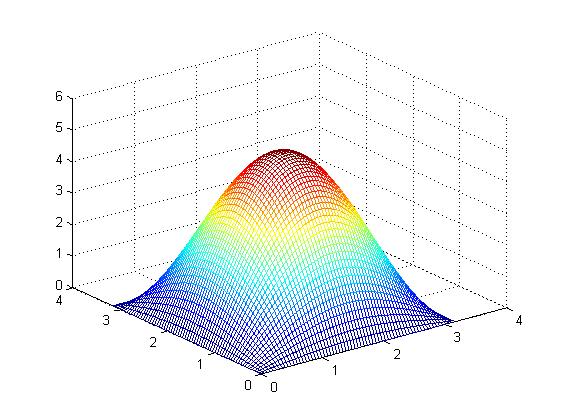} \\
  (b)
\end{minipage}
\caption{Exact solution (a) and approximate solution (b) obtained by algorithm \eqref{eq:x}$-$\eqref{eq:y} at $T=0.5$ with mesh $h_1=h_2=\frac{\pi}{64}$,
$\Delta\beta=\frac{1}{64}$, and $\tau=\frac{1}{4096}$.}\label{fig:solution}
\end{figure}

In Table \ref{table:a}, the numerical accuracy of difference scheme \eqref{eq:x}$-$\eqref{eq:y} in time is
recorded.
Let the step sizes $h_1$, $h_2$, and $\Delta\beta$ be fixed and small enough such that
the dominated error arise from the approximation of the time derivatives.
Varying the step sizes in time, the numerical errors in discrete both $L^\infty$ and $L^2$ norms
and the associated convergence orders are shown in this table respectively, which can be found in agreement with the theoretical analysis.

In Table \ref{table:b}, we take the fixed and small enough step sizes in space, and adopt an optimal step
size ratio in time and distributed order. As $\Delta\beta$ and $\tau$ vary, we compute the errors
and convergence orders listed in the table, which indicates that the convergence order in time and distributed order
are about one and two, respectively.

Table \ref{table:c} displays the computational results with an optimal step size ratio in time, space and
distributed order. We can conclude from this table that the convergence orders with respect to time, space and distributed order
are approximately one, two and two, respectively, which is in good agreement with our theoretical results analyzed in Section \ref{sec:analysis}.

\begin{table}[h!]
\caption{Errors and convergence orders for Example~\ref{exam:a} in temporal direction
 with $h_1=h_2=\frac{\pi}{500}$ and $\Delta\beta=\frac{1}{160}$.}\label{table:a}
  \centering
  \begin{tabular}{ccccc}
    \toprule
    $\tau$ & $\big\|e^N\big\|_{L^\infty} $  & Order   & $\big\|e^N\big\|_{L^2}$ & Order    \\
    \midrule
      1/10   & 0.0839      &      - & 0.1225        &      - \\
      1/20   & 0.0439      & 0.9344 & 0.0634        & 0.9502 \\
      1/40   & 0.0227      & 0.9515 & 0.0326        & 0.9596 \\
      1/80   & 0.0117      & 0.9526 & 0.0167        & 0.9650 \\
      1/160  & 0.0059      & 0.9877 & 0.0085        & 0.9743 \\
    \bottomrule
  \end{tabular}
\end{table}

\begin{table}[h!]
\caption{Errors and convergence orders for Example~\ref{exam:a} with an optimal step size ratio for
$\tau$ and $\Delta\beta$, and $h_1=h_2=\frac{\pi}{500}$.}\label{table:b}
  \centering
  \begin{tabular}{cccccc}
    \toprule
    $\tau$ & $\Delta\beta $  &  $\big\|e^N\big\|_{L^\infty}$  &  Order  &  $\big\|e^N\big\|_{L^2}$  &  Order \\
    \midrule
      1/100   & 1/10   &  0.0093        &      - & 0.0133        &      -    \\
      1/400   & 1/20   &  0.0024        & 1.9542 & 0.0034        & 1.9678    \\
      1/1600  & 1/40   &  6.0481e-04    & 1.9885 & 8.6411e-04    & 1.9762    \\
      1/6400  & 1/80   &1.4751e-04      & 2.0357 & 2.1076e-04    & 2.0365    \\
    \bottomrule
  \end{tabular}
\end{table}

\begin{table}[h!]
\caption{Errors and convergence orders for Example~\ref{exam:a} with an optimal step size ratio
for $\tau$, $h_1$, $h_2$, and $\Delta\beta$.}\label{table:c}
  \centering
  \begin{tabular}{ccccccc}
    \toprule
    $\tau$ & $h_1=h_2$  & $\Delta\beta$   & $\big\|e^N\big\|_{L^\infty}$ & Order   & $\big\|e^N\big\|_{L^2}$ & Order \\
    \midrule
      1/64       & $\pi/2$    & 1/8      & 0.4602        &      -   & 0.7230        &      -\\
      1/256      & $\pi/4$    & 1/16     & 0.1195        & 1.9453   & 0.1689        & 2.0978\\
      1/1024     & $\pi/8$    & 1/32     & 0.0301        & 1.9892   & 0.0426        & 1.9872\\
      1/4096     & $\pi/16$   & 1/64     & 0.0075        & 2.0048   & 0.0107        & 1.9932\\
      1/16384    & $\pi/32$   & 1/128    & 0.0019        & 1.9809   & 0.0027        & 1.9866\\
      1/65536    & $\pi/64$   & 1/256    & 4.7098e-04    & 2.0123   & 6.6801e-04    & 2.0150\\
    \bottomrule
  \end{tabular}
\end{table}


\section{Conclusion}\label{sec:conclusion}
In this paper, we construct efficient numerical scheme for solving two-dimensional time-fractional
wave equation of distributed-order with a
nonlinear source term, and provide the theoretical analysis on stability and convergence by the discrete energy method.
Numerical results are provided by figures and tables, which show the algorithm proposed in this work is
effective and feasible.
In the future work,
the promotion of computational efficiency
will be considered so that the more complicated problems can be handled.

\section*{Acknowledgements}
This research was supported by National Natural Science Foundations of China (No.11471262).
The authors would like to express their gratitude to the referees for their very helpful
comments and suggestions on the manuscript.

\bibliography{Reference}
\end{document}